\documentclass[a4paper,12pt]{article}
\usepackage[utf8]{inputenc}
\usepackage[american]{babel}
\usepackage{amsmath}
\usepackage{amssymb}
\usepackage{amsthm}
\usepackage{xcolor}
\usepackage{tikz}
\usepackage{float}
\usepackage{url}
\usepackage[hidelinks]{hyperref}
\hypersetup{
    pdftitle={On Scarf's Theorem for Generalized Cooperative Games},
    pdfauthor={Mikhail V. Bludov and Oleg R. Musin}
}

\newtheorem{theorem}{Theorem}[section]
\newtheorem{lemma}[theorem]{Lemma}

\newtheorem{proposition}[theorem]{Proposition}

\theoremstyle{definition}
\newtheorem{definition}{Definition}[section]
\newtheorem{exmp}{Example}[section]

\theoremstyle{remark}
\newtheorem{remark}{Remark}

\newcommand{\bs}{\mathop{\rm BS}\nolimits}
\newcommand{\cone}{\mathop{\rm cone}\nolimits}
\newcommand{\conv}{\mathop{\rm conv}\nolimits}

\title{On Scarf's Theorem for Generalized Cooperative Games}

\author{Mikhail V. Bludov\textsuperscript{a,b}\thanks{Corresponding author. Email: \href{mailto:bludov.mv@phystech.edu}{\nolinkurl{bludov.mv@phystech.edu}}. Postal address: 9 Institutskiy Lane, Dolgoprudny, Moscow Region 141700, Russian Federation.}
    \and Oleg R. Musin\textsuperscript{c}}

\date{%
    \small
    \textsuperscript{a}HSE University, 11 Pokrovsky Bulvar, Moscow 109028, Russian Federation\\
    \textsuperscript{b}Higher School of Mathematics, Moscow Institute of Physics and Technology\\
    9 Institutskiy Lane, Dolgoprudny, Moscow Region 141700, Russian Federation\\
    \textsuperscript{c}School of Mathematical and Statistical Sciences,\\
    The University of Texas Rio Grande Valley, 1 West University Boulevard,\\
    Brownsville, TX 78520, USA\\
    \href{mailto:bludov.mv@phystech.edu}{\nolinkurl{bludov.mv@phystech.edu}}
    \qquad
    \href{mailto:oleg.musin@utrgv.edu}{\nolinkurl{oleg.musin@utrgv.edu}}%
}

\begin{document}

\maketitle

\begin{abstract}
 In this paper, we study a generalization of cooperative games with non-transferable utility. Our model is motivated by Shapley’s proof of Scarf’s theorem, which revealed a deep connection between KKM-type theorems and cooperative games. In recent decades, further developments have established deeper connections between algebraic topology and combinatorial KKM-type fixed-point theorems. We explore the consequences of these developments for cooperative game theory.

In our model, coalitions are replaced by firms, each represented by a resource vector and a set of feasible payoff vectors in a common payoff space. We define the core and fractional core for these games and relate their nonemptiness to homotopy invariants of covers associated with the game. Consequently, we obtain a Scarf-type theorem for generalized cooperative games. This framework allows us to model cooperative games with threats and political games.

\end{abstract}

\begin{center}
\begin{minipage}{0.85\textwidth}
\end{minipage}
\end{center}
\clearpage

\sloppy

\noindent\textbf{Notation.}
In this paper, we use the following notation.
\begin{itemize}
\item We write
\([n]=\{1,\dots,n\}\) for the set of \(n\) elements (players). The power set
\(2^{[n]}\) of \([n]\) is the set of all subsets of \([n]\), including the
empty set and \([n]\) itself.

\item We also write \(N=[n]\) for the player set.
For \(S\subseteq N\), we denote its complement by
\(S^{\mathrm{op}}=N\setminus S\).

\item Unless stated otherwise, subsets \(S\subset [n]\) and index sets \(I\subset[m]\)
are nonempty. A {\em proper} subset is a subset that is not equal to the
whole set. The symbol \(\subset\) allows equality.

\item For \(S\subset [n]\), the vector
\(\mathbf{1}_S\) is the characteristic vector of \(S\) in \(\mathbb{R}^n\),
and \(\mathbb{R}^S\) denotes the subspace of \(\mathbb{R}^n\) spanned by
\(\{\mathbf{1}_{\{i\}}\}_{i\in S}\). We write
\(e_i=\mathbf{1}_{\{i\}}\) for the standard basis vectors.

\item The nonnegative orthants in \(\mathbb{R}^n\) and \(\mathbb{R}^S\)
are denoted by \(\mathbb{R}^n_+\) and \(\mathbb{R}^S_+\), respectively.

\item For a set \(A\) in a Euclidean space, \(\conv(A)\) denotes its convex
hull and \(\cone(A)\) denotes its conical hull.

\item The simplex \(\Delta^{n-1}\) is the convex hull
\(\conv(\mathbf{1}_{\{1\}},\dots,\mathbf{1}_{\{n\}})\).

\item For \(S\subset[n]\), the face \(\Delta_S\) is the convex hull of
\(\{\mathbf{1}_{\{i\}}\}_{i\in S}\).

\item For two vectors \(x,y\in\mathbb{R}^n\), we write \(y\leq x\) if and
only if \(y_i\leq x_i\) for all \(i=1,\dots,n\).

\item \(S^k\) denotes the \(k\)-dimensional sphere and \(B^k\) the
\(k\)-dimensional ball; in particular, \(S^k=\partial B^{k+1}\).

\item For topological spaces \(X\) and \(Y\), \([X,Y]\)
denotes the set of free homotopy classes of continuous maps from \(X\) to \(Y\).
\end{itemize}

\medskip

\section{Introduction}

Scarf's theorem \cite{Scarf} is a fundamental result in cooperative game theory. It states that every balanced cooperative NTU game has a nonempty core.
Shapley \cite{Sh} gave a topological proof using his KKMS theorem. KKM-type covering theorems have been developed and extended in several directions.
Musin introduced homotopy invariants of covers \cite{MusH}, and Bludov
developed the theory of these invariants relative to balanced configurations
of points \cite{Blu}.
In this paper, we explore how these developments in KKM covering theory yield new results in cooperative game theory.
In particular, we introduce a new class of generalized cooperative games and establish a Scarf-type theorem for this class.
This framework accommodates political games and cooperative games with threats.

The introduction is organized as follows. In Sections 1.1 and 1.2,
we recall the basic definitions and results from cooperative game theory. In
Section 1.3, we introduce the concept of the fractional core. In Section
1.4, we briefly discuss the algebraic topology behind the KKMS theorem and
give references to some extensions of the KKMS theorem. In Section 1.5, we
introduce the main definitions of this paper and formulate the main results.

\subsection{Bondareva--Shapley theorem.}

We begin with the classical model of cooperative games with transferable
utility.

\begin{definition}\label{TU} A cooperative game with {\em transferable utility} (TU game) is a pair $([n], \nu)$, where $[n]=\{1,\dots,n\}$ is a set of players and the {\em characteristic function} $\nu:2^{[n]}\rightarrow \mathbb{R}$ satisfies $\nu(\emptyset)=0$. Subsets $S\subset [n]$ are called {\em coalitions.} 
\end{definition}

The characteristic function $\nu$ represents the utility that a coalition can achieve.

\begin{definition}\label{payoff}
A vector $x=(x_1,...,x_n)\in \mathbb{R}^n$ is a \emph{payoff vector} if $  \sum_{i=1}^n x_i=\nu([n]).$ 
The set of all payoff vectors is denoted by $E(\nu)$. 

The \emph{core} of the game $([n],\nu)$ is the set of payoff vectors $x\in E(\nu)$ such that
\[
    \sum_{i\in S}x_i\geq \nu(S)
    \quad \text{for every } S\subset [n].
\]
\end{definition}

In cooperative game theory, the core is a set of payoff vectors (allocations) where no subset of players (a ``coalition'') can do better by breaking away and acting alone. It ensures stability so that all players are motivated to stay in the grand coalition.



\begin{definition}\label{FamBal}
A family $\Phi=\{S_1,\ldots,S_m\}$ of subsets of $[n]$ is \emph{balanced} if there exist non-negative weights $\{\lambda_{S_k}\}$ such that
\[
    \sum_{k=1}^m \lambda_{S_k}\mathbf{1}_{S_k}=\mathbf{1}_{[n]}.
\]
\end{definition}

This definition has a geometric interpretation. Let $\Delta_{[n]}=\operatorname{conv}(\mathbf{1}_{\{1\}},\dots,\mathbf{1}_{\{n\}})$ be the standard simplex, and let $\Delta_S$ denote the face corresponding to a coalition $S$. If
\[
    c_S=\frac{1}{|S|}\mathbf{1}_S
\]
is the barycenter of $\Delta_S$, then a family $\Phi$ is balanced if and only if
\[
    c_{[n]}\in \operatorname{conv}\{c_S\}_{S\in\Phi}.
\]

\begin{definition}\label{TUBal} We say that a TU game $([n],\nu)$ is \emph{balanced} if for every balanced family $\Phi$ with weights $\{\lambda_S\}$ we have 
\[
    \sum_{S\in\Phi}\lambda_S\nu(S)\leq \nu([n]).
\]
\end{definition}

Actually, the core of a TU game may be empty. The Bondareva--Shapley theorem gives necessary and sufficient conditions for core nonemptiness in terms of balancedness.

\begin{theorem}[Bondareva \cite{Bon}, Shapley \cite{Sh67}]\label{BShTh}
A cooperative TU game has a nonempty core if and only if it is balanced.
\end{theorem}

Bondareva derived this theorem from linear programming duality\cite{Bon}. Shapley gave a short proof
using convex cones \cite[Theorem~1]{Sh67}.
 

\subsection{Scarf's theorem}

In the TU setting, a coalition \(S\) is assigned a single value \(\nu(S)\),
and the members of \(S\) are assumed to be able to redistribute this value
among themselves without restrictions. Thus the efficient payoff vectors of
\(S\) lie on the hyperplane \(\sum_{i\in S}x_i=\nu(S)\). This assumption is
often too restrictive. Imagine, for instance, a situation where two players
work together, but only one of them can receive the payment. If he wants to
transfer part of this payment to the other player, he has to pay a transaction
fee to the bank. Thus the resulting set of payoff vectors may form a nonlinear
subset of \(\mathbb{R}^2\). This example leads naturally to cooperative games
with non-transferable utility.

\begin{definition}\label{NTU}
Let \( [n]=\{1,\dots,n\} \) be the set of players. A cooperative game with \emph{non-transferable utility} (NTU game) is a pair \( ([n],V) \), where
\[
V=\bigl(V(S)\bigr)_{\varnothing\neq S\subseteq[n]},
\qquad V(S)\subseteq\mathbb{R}^S,
\]
satisfying the following conditions:
\begin{itemize} 

\item \(V(S)\) is closed, nonempty, and proper for every nonempty \(S\subseteq[n]\).

\item \(V(S)\) is \emph{comprehensive}:
\[
    V(S)-\mathbb{R}^S_+\subseteq V(S).
\]
Equivalently, if \(x\in V(S)\) and \(x'\leq x\), then \(x'\in V(S)\).

\item Write \(V(\{i\})=(-\infty,a_i]\), where
\(a_i=\max V(\{i\})\). The individually rational part
\[
    V(S)\cap\{x\in\mathbb{R}^S:x_i\geq a_i\text{ for every }i\in S\}
\]
is bounded for every nonempty \(S\subseteq[n]\).

\end{itemize}
\end{definition}

\medskip

For every nonempty subset \( S \subset [n] \), define the cylinder
\( U(S)=V(S)+\mathbb{R}^{[n]\setminus S}\subset\mathbb{R}^n. \) For a coalition \(S\), the subset \(V(S)\) represents the set of feasible payoff vectors. If \(x\in \operatorname{int}(V(S))\), where the interior is taken in \(\mathbb{R}^S\), then the coalition \(S\) can strictly improve this payoff. Thus \(\operatorname{int}(U(S))\) represents the set of allocations blocked by coalition \(S\).

\medskip

A solution of the NTU game is an allocation vector \(x\in V([n])\) such that no coalition can strictly improve its payoff by acting alone. This leads to the following definition of the core. 
\begin{definition}
The \textit{core} \(C(V)\) of the game \(([n],V)\) is the set
\[
C(V) = V([n]) \setminus \bigcup\limits_{S \subset [n]} \operatorname{int}(U(S)).
\]
\end{definition}
\medskip

Sometimes it is natural to assume that the grand coalition \([n]\)
can achieve every payoff vector feasible under a partition of the
players into coalitions. In this case, for every partition
\([n]=S_1\sqcup\cdots\sqcup S_k\), we assume \( \bigcap_{j=1}^k U(S_j)\subseteq V([n]).\)

Balancedness can be extended from TU games to NTU games. An NTU game is \emph{balanced} if for every balanced family \(\Phi\) we have \(\bigcap_{S \in \Phi} U(S) \subset V([n]).\)
We recover the following version of Scarf's theorem
\cite{Scarf} in Section~4.1.

\begin{theorem}[Scarf]\label{Scarf}
Every balanced cooperative NTU game has a nonempty core.
\end{theorem}

Scarf's original proof uses a finite pivoting algorithm based on the
Lemke--Howson procedure \cite[Sections~3--4]{Scarf}.

TU games form a particular class of NTU games. Namely, to a TU
game \(([n],\nu)\) one associates the NTU game given by
\[
    V_\nu(S)=\left\{y\in\mathbb{R}^S\mid \sum_{i\in S}y_i\leq \nu(S)\right\}.
\]
The singleton ceilings are \(a_i=\nu(\{i\})\).
Then \(U(S)=V_\nu(S)+\mathbb{R}^{[n]\setminus S}\), and the condition
\(x\notin\operatorname{int}(U(S))\) is exactly the usual TU inequality
\(\sum_{i\in S}x_i\geq\nu(S)\).
For such games, Scarf's balancedness condition is equivalent to the usual
balancedness condition from the Bondareva--Shapley theorem. Thus, on this
subclass, the same balancedness condition is both sufficient and necessary for
nonemptiness of the core. This equivalence does not extend to general NTU
games: Scarf's theorem gives only a sufficient condition. This gap led to
refinements of balancedness. Billera introduced \(\pi\)-balancedness
\cite{Billera}. When the grand coalition's individually rational
payoff set is nonempty and compact, Predtetchinski and Herings characterize
core nonemptiness by \(\Pi\)-balancedness for a suitable correspondence
\(\Pi\) \cite[Theorems~3 and~4]{Pred}.

\subsection{Fractional form of Scarf's theorem}
Scarf's theorem gives a sufficient condition for nonemptiness of the core,
but the core itself may be empty. We therefore consider a more flexible
solution concept, called fractional core. Related notions appear in work by
Danilov \cite{Dan} and by Biró and Fleiner \cite{Fleiner}.

\medskip

\begin{definition}

Let \(\Phi\) be a balanced family with weights \(\{\lambda_S\}\). A point \(x\in \mathbb{R}^n\) is \emph{admissible} with respect to \(\Phi\) if
\[
    x \in U(S) \quad \text{whenever } \lambda_S>0.
\]
The \emph{fractional core} of the game is the set of points \(x\) that are admissible with respect to some balanced family and such that
\[
    x \notin \operatorname{int}(U(S))
    \qquad \text{for all } S\subset [n].
\]

\end{definition}

Thus, a point from the fractional core is stable against deviations by coalitions, but it is allowed to arise from a balanced fractional use of several coalitions rather than from the grand coalition alone.

\medskip

We now illustrate the difference between the core and the fractional core in a
simple TU game, using the identification of TU games with the corresponding
NTU games described in the previous subsection.

\begin{exmp}
Consider a 3-player cooperative game where players face losses:
\[
\begin{aligned}
    \nu(\{1\}) &= -10, \quad \nu(\{2\}) = -15, \quad \nu(\{3\}) = -20,\\
    \nu(\{1,2\}) &= -22, \quad \nu(\{1,3\}) = -28, \quad \nu(\{2,3\}) = -32,\\
    \nu(\{1,2,3\}) &= -35.
\end{aligned}
\]
The allocation \((-8,-12,-15)\) belongs to the core, but if the grand
coalition's utility is changed to \(\nu(\{1,2,3\})=-100\), then the core
becomes empty. However, for the balanced family
\[
    \Phi=\{\{1,2\},\{1,3\},\{2,3\}\},\qquad
    \lambda_{\{1,2\}}=\lambda_{\{1,3\}}=\lambda_{\{2,3\}}=\frac12,
\]
the vector \(x=(-9,-13,-19)\) is admissible and does not lie in the interiors
of the coalitional cylinders \(U(S)\). Hence \(x\) belongs to the fractional
core. Thus the fractional core may still be nonempty even when the core is
empty.
\end{exmp}

Following Danilov's resource interpretation \cite{Dan}, assume that each
player possesses one unit of an individual resource, such as time. A
player may distribute this time among several coalitions, and a
coalition \(S\) operates when each of its members contributes the same amount.
Thus a family of coalitions is balanced when the players can fully allocate
their resources. We view \(U(S)\) as the set of global payoff vectors feasible
for coalition \(S\). An admissible point is a global payoff vector feasible for
every coalition from some balanced family. Thus a fractional-core point is
a stable payoff vector supported by a full allocation of
the players' time among coalitions. In the capacitated model of Biró and
Fleiner, coalitional activities can require different contributions from
different players, and fractional core elements record activity levels
\cite[Sections~2.2 and~3.3]{Fleiner}.
 
\begin{theorem}
\label{SDan}
Every NTU game has a nonempty fractional core.
\end{theorem}

This theorem follows from Scarf's proof \cite[Sections~3--6]{Scarf}.
Danilov's equilibrium states pair a fractional-core payoff with its
balancing weights \cite{Dan}. In Section~4.1, we sketch a proof using
the KKMS theorem.

\subsection{KKMS theorem.}
The fractional form of Scarf's theorem
(Theorem~\ref{SDan}) is of topological origin. To see this, we turn to
Shapley's approach via coverings. In \cite{Sh}, Shapley proved his famous
KKMS theorem and deduced Scarf's theorem from it.
We use the following form of the KKMS theorem.

\begin{theorem}[KKMS]
\label{KKMStheorem}
Let \(\Delta^{n-1}\) be a simplex whose vertices are
indexed by \([n]\). Let
\(\mathcal{F}=(F_S)_{\varnothing\neq S\subseteq[n]}\)
be a closed cover satisfying the KKMS face condition
\begin{equation}\label{KKMSFaceCondition}
    F_S\cap\Delta_T=\varnothing\qquad\text{whenever }S\nsubseteq T.
\end{equation}
Then there exists a balanced family \(\Phi\) such that
\(\bigcap_{S\in\Phi}F_S\neq\varnothing\).
\end{theorem}

Shapley's construction associates a KKMS cover to an NTU
game. For \(n\geq2\), the KKMS face condition gives a boundary map of
degree \(1\), as shown in Section~\ref{sec:kkms-borsuk-ulam}.

\medskip

The geometry and topology of simplicial complexes associated with simple
games illustrate another link between cooperative game theory and
topological combinatorics \cite{Polytopality}.

Since Shapley's work, the KKM and KKMS theorems have been extended in
several directions. Komiya extended the KKMS theorem to polytopes
\cite{Komiya}. Musin developed homotopy invariants of covers and applied
topological methods to extensions of the KKM, Sperner, and Tucker lemmas
\cite{MusH,MusBC,MusSpT}. Musin and Wu used cobordism classes to study
covers of spheres \cite{MusWu}.

Bludov developed extension theorems for covers relative to balanced
configurations of points \cite{Blu}. Bludov's construction pairs a cover
of a sphere or a more general space with a separately specified
configuration of points. A cover with no convex balanced intersections
determines a homotopy class of maps to the simplicial complex of
unbalanced subfamilies. If this class is nontrivial on the boundary of a ball, every extension of the cover to the ball has a convex balanced
intersection. This construction motivates the generalized cooperative
games introduced below. Section~2 gives the precise definitions.

\subsection{Generalized Cooperative Games}

We now introduce the formal definition of generalized games considered in this
paper.

\begin{definition}\label{GNTP}
A \emph{generalized NTU cooperative game} is a triple \((U,V,r)\).
The firms are identified with the labels \(i\in[m]\).
For each firm \(i\), let \(U_i\subset\mathbb{R}^n\) be its feasible set and
\(v_i\in\mathbb{R}^d\) its resource vector. Write \(U=(U_i)_{i\in[m]}\)
and \(V=(v_i)_{i\in[m]}\). The triple satisfies the following conditions.
\begin{enumerate}
    \item Each \(U_i\) is closed, proper, nonempty and comprehensive, i.e.
    \(U_i-\mathbb{R}^n_+\subseteq U_i\).
    \item Each \(v_i\) satisfies \(\langle v_i,\mathbf{1}_{[d]}\rangle>0\);
    the cone \(\cone(V)\) is full-dimensional in \(\mathbb{R}^d\), and
    \(r\neq0\), \(r\in\operatorname{int}(\cone(V))\).
\end{enumerate}
Distinct firms may have the same resource vector or feasible set.
\end{definition}

For brevity, we will refer to generalized NTU cooperative games simply as
games.

\medskip

\begin{definition}
An index set \(I\subset[m]\) is \emph{\(r\)-balanced} if
\(r\in\cone\{v_i:i\in I\}\), or equivalently, if
\(r=\sum_{i\in I}\lambda_i v_i\) for some \(\lambda_i\geq0\).
The family of all such index sets is denoted by \(\bs(V,r)\).
\end{definition}

\medskip

\begin{definition}
A point \(x\in\mathbb{R}^n\) is \emph{admissible} if, for some \(I\in\bs(V,r)\) with balancing weights \((\lambda_i)_{i\in I}\), we have
\[
    x\in U_i \qquad \text{whenever } \lambda_i>0.
\]
Equivalently, \(x\in\bigcap_{i\in I}U_i\) for some
\(I\in\bs(V,r)\).
The \emph{fractional core} of the game is the set of admissible points \(x\) such that
\[
    x\notin \operatorname{int}(U_i)
    \qquad \text{for all } i=1,\dots,m.
\]
\end{definition}

\medskip

If firm \(m+1\) has resource vector \(v_{m+1}=r\), we distinguish it as the grand
firm, and \(U_{m+1}\) plays the role of the feasible payoff set of the grand firm. In
this case the \emph{core} of the game is defined by
\[
    C(U,V,r)=U_{m+1}\setminus \bigcup_{i=1}^{m+1}\operatorname{int}(U_i).
\]

We now consider a simple example showing that, in generalized games, even the
fractional core may be empty.

\begin{exmp}
Let the resource vector be \(r=(1,1)\), and let there be two firms
\(v_1=(1,0)\) and \(v_2=(0,1)\). The payoff space is \(\mathbb{R}\). Suppose
that for the first firm only payoff vectors \(x\leq 0\) are feasible, while
for the second firm all payoff vectors \(x\leq1000\) are feasible. Thus
\[
    U_1=(-\infty,0],\qquad U_2=(-\infty,1000].
\]
The only \(r\)-balanced index set is \(\{1,2\}\). Hence every admissible point
lies in \(U_1\cap U_2=(-\infty,0]\), which is contained in
\(\operatorname{int}(U_2)\). Therefore the fractional core is empty.
\end{exmp}

Although the example is formal, it may be read as a simple hiking problem:
one friend cannot climb above height \(0\), while the other can climb up to
height \(1000\). If they go together, the second friend is disappointed; if
only the second friend goes, the first friend's resource is unused.

The definition above is purely formal, and we do not fix any particular
economic interpretation. One possible intuition is to view \(v_i\) as the
resource requirement of firm \(i\), and \(U_i\) as the set of payoff vectors
feasible for this firm; points in \(\operatorname{int}(U_i)\) are then
allocations blocked by it. An \(r\)-balanced family may be regarded as an
active family consuming the whole resource vector. The fractional-core problem
asks whether some active family admits a payoff vector feasible for all its
firms and not strictly blocked by any firm.

The observation that firm \(i\) blocks every allocation in \(\operatorname{int}(U_i)\) gives rise to a new class of cooperative games with threats, in which players or coalitions may impose vetoes or threats on allocations to other players. We explore this class in Section~\ref{sec:threats}.
 
Every generalized game determines a closed cover
\(F(U)=(F(U_i))_{i\in[m]}\) of the payoff hyperplane
\(H_0=\{x\in\mathbb{R}^n:\sum_{j=1}^n x_j=0\}\). The cover is constructed in Section~\ref{sec:game-cover}.
The balanced intersection set of the cover is
\[
    Z(U,V,r)=
    \bigcup_{I\in\bs(V,r)}\bigcap_{i\in I}F(U_i).
\]
The game is \emph{balanced intersecting} if
\(Z(U,V,r)\neq\varnothing\).

\begin{theorem}[Main Theorem: fractional-core correspondence]
\label{GenTheorem}
There is a homeomorphism between \(Z(U,V,r)\)
and the fractional core of the game. In particular, the fractional core
is nonempty if and only if the game is balanced intersecting.
\end{theorem}

We also propose a way to detect balanced intersections via homotopy invariants.
Informally, a sufficient condition for a game to be homotopically nontrivial is the existence of a sphere \(S^{n-2}\subset H_0\), with \(n\geq2\), on which the restricted cover has a defined, nontrivial homotopy class with respect to the normalized resource pair \((\bar V,\bar r)\). The precise definition is given in Section~3.4.

\begin{theorem}
\label{HomNontrivGameThm}
Every homotopically nontrivial game \((U,V,r)\) has a nonempty fractional core.
\end{theorem}
For example, every classical NTU game with at least two players is homotopically nontrivial. We also give examples within this generalized framework of both homotopically trivial and homotopically nontrivial games.

For balanced games, we obtain a Scarf-type result. Namely, suppose that one of the firms is the grand firm \(v_{m+1}=r\). The game \((U,V,r)\) is called \emph{balanced} if, for every \(I\in\bs(V,r)\),
\[
    \bigcap_{i\in I}U_i\subset U_{m+1}.
\]

\begin{theorem}[Scarf's theorem for generalized cooperative games]\label{GenScTh}
Let \((U,V,r)\) be a balanced game with \(V=(v_i)_{i\in[m+1]}\), where \(v_{m+1}=r\). Then the core of \((U,V,r)\) is nonempty if the game is homotopically nontrivial.
\end{theorem}

For a balanced game containing the grand firm, the core and fractional
core coincide. 

Cooperative games with threats are one of the main applications of our
general framework. In Section~\ref{sec:threats}, we give examples showing
that threats can change the boundary degree of the associated cover and
can even make the fractional core empty. Under suitable convexity and
symmetry conditions, a Borsuk--Ulam type generalization of the KKMS theorem
guarantees that the fractional core is nonempty.

\subsection*{Structure of the paper}

The paper is structured as follows. In Section 2, we provide the necessary
topological preliminaries. In Section 2.1, we discuss the geometric definition
of balanced sets. In Section 2.2, we define the homotopy class of a cover
relative to a balanced configuration. In Section 2.3, we define
homotopically nontrivial covers. In Section 2.4, we discuss the local index
of a convex balanced intersection of a cover.

In Section 3, we prove the main theorem. More precisely, in Section 3.1, we
define the cover associated with a game. In Section 3.2, we explain the
connection between fractional cores and balanced intersections of this cover.
In Section 3.3, we introduce the normalization of the resource vectors. In
Section 3.4, we prove the homotopy criterion and its Scarf-type consequence. In Section 4,
we discuss examples and applications. In Section 4.1, we recover the
classical Scarf theorem. In Section 4.2, we introduce games of degree \(k\)
using the local index of connected components of the convex balanced intersection set. In
Section 4.3, we discuss an example of a game whose fractional core is
guaranteed by a map from \(S^3\) to \(S^2\) with nontrivial Hopf invariant.
Section~4.4 treats small perturbations when the resource
dimension exceeds the payoff dimension, and Section~4.5 presents a model of
political games.
Section~\ref{sec:threats} treats cooperative games with
threats. Section~\ref{sec:further} discusses further research directions.

\section{Topological Preliminaries}

This section recalls the topological notions used below. All covers are
finite and indexed by the same set \([m]\). The construction of homotopy
classes of covers is due to Musin \cite{MusH}; the results involving
balanced sets are taken from \cite{Blu}.

\subsection{Balanced configurations}

For a finite indexed configuration \(V=(v_i)_{i\in[m]}\)
in \(\mathbb{R}^q\) and a point
\(r\in\operatorname{relint}\conv(V)\), the \emph{rank} of the pair
\((V,r)\) is the dimension of the vector space
\(\operatorname{span}\{v_i-r\}_{i=1}^m\). We write \(d\) for the rank of \((V,r)\).
For \(I\subset[m]\), put \(V_I=(v_i)_{i\in I}\). The index set \(I\), or the
corresponding indexed subfamily \(V_I\), is called
\emph{convex \(r\)-balanced} if \(r\in\conv(V_I)\). The family of all
convex \(r\)-balanced subsets is denoted by
\(\conv\bs(V,r)\).

Unless stated otherwise, throughout this section we fix such a pair \((V,r)\).
Write \(A(V,r)=\operatorname{aff}(V)\) and
\(E(V,r)=A(V,r)-r\), and let \(S(V,r)\) be the unit sphere in the
Euclidean vector space \(E(V,r)\).

Two pairs \((V,r)\) and \((V',r')\), with the same index set \([m]\), are
called \emph{\(\conv\bs\)-equivalent} if they have the same convex balanced
subsets of indices. By Corollary~3.7 of \cite{Blu}, two
\(\conv\bs\)-equivalent pairs have the same rank.

\subsection{Homotopy class of a cover}

For a cover \(\mathcal{F}=(F_i)_{i\in[m]}\) of \(X\) and a
nonempty subset \(A\subseteq X\), its restriction to \(A\) is the cover
\(\mathcal{F}|_A=(F_i\cap A)_{i\in[m]}\) of \(A\) with the subspace
topology. The restriction retains all labels, including those for which
\(F_i\cap A=\varnothing\).

Let \(X\) be a nonempty subset of a Euclidean space, with the
relative topology, and let \(\mathcal{F}=\{F_1,\dots,F_m\}\) be an open
cover of \(X\). The
\emph{nerve} \(N(\mathcal{F})\) is the simplicial complex on the index set
\([m]\) such that a nonempty set \(I\subset[m]\) is a simplex precisely
when \(\bigcap_{i\in I}F_i\neq\emptyset\).

A \emph{partition of unity subordinate to} \(\mathcal{F}\) is a family of
continuous functions
\(\Phi=\{\varphi_1,\dots,\varphi_m\}\),
\(\varphi_i:X\to[0,1]\), such that
\(\sum_{i=1}^m\varphi_i(x)=1\) for all \(x\in X\), and
\(\varphi_i(x)=0\) whenever \(x\notin F_i\). Since \(X\) is a subset of a
Euclidean space, such partitions of unity exist for every finite open cover.

Suppose first that the nerve \(N(\mathcal{F})\) contains no simplex from
\(\conv\bs(V,r)\). Define
\[
    \rho_{\Phi,V}\colon X\longrightarrow A(V,r),
    \qquad \rho_{\Phi,V}(x):=\sum_{i=1}^m\varphi_i(x)v_i.
\]
Then \(\rho_{\Phi,V}(x)\neq r\) for every \(x\in X\). Indeed, if
\(\rho_{\Phi,V}(x)=r\), then the set
\(\{i\mid \varphi_i(x)>0\}\) would be a simplex of \(N(\mathcal{F})\) and
would be convex \(r\)-balanced.

The space \(E(V,r)\) has dimension \(d\), so
\(S(V,r)\simeq S^{d-1}\) for \(d\geq1\). Radial projection from \(r\) gives
the map \(g_{\Phi,V,r}:X\longrightarrow S(V,r)\) defined by
\[
    g_{\Phi,V,r}(x)=
    \frac{\rho_{\Phi,V}(x)-r}{\|\rho_{\Phi,V}(x)-r\|}.
\]

A homotopy class \([g]\in[X,S(V,r)]\) is \emph{trivial} if it contains a constant map.

\begin{definition}
\label{VHomCov}
If \(N(\mathcal{F})\) contains no simplex from \(\conv\bs(V,r)\), the
\emph{homotopy class of \(\mathcal{F}\) with respect to \((V,r)\)} is the
homotopy class of the map \(g_{\Phi,V,r}\). It is denoted by
\([\mathcal{F}_{(V,r)}]:=[g_{\Phi,V,r}]\in[X,S(V,r)]\). If
\(N(\mathcal{F})\) contains a simplex from \(\conv\bs(V,r)\), we write
\([\mathcal{F}_{(V,r)}]=N/A\).
\end{definition}

For example, if \(d=0\), then \(v_i=r\) for every \(i\), so each singleton index
set is convex \(r\)-balanced. Since \(X\) is nonempty, some \(F_i\) is
nonempty, and hence \([\mathcal{F}_{(V,r)}]=N/A\).

The class \([\mathcal{F}_{(V,r)}]\) does not depend on the choice of
partition of unity \(\Phi\) \cite[Lemma~1.6]{MusH}.

For a finite closed cover \(\mathcal{F}=\{F_1,\dots,F_m\}\) of \(X\), an
\emph{open enlargement} is an open cover
\(\mathcal{G}=\{G_1,\dots,G_m\}\) such that \(F_i\subset G_i\) for all
\(i\) and \(N(\mathcal{G})=N(\mathcal{F})\).
Such enlargements exist for finite closed covers of subsets of Euclidean
spaces, and all of them yield the same homotopy class
\cite[Lemma~2.4]{Blu}. We define the homotopy class of a closed cover by
\([\mathcal{F}_{(V,r)}]:=[\mathcal{G}_{(V,r)}]\) for any such enlargement
\(\mathcal{G}\).

Choose an orientation of \(E(V,r)\), and orient \(S(V,r)\)
as the boundary of the unit ball in \(E(V,r)\).
Let \(M\) be a compact oriented \(d\)-manifold with nonempty boundary, and give \(\partial M\) its induced orientation. For \(d\geq2\), a homotopy class \([g]\in[\partial M,S(V,r)]\) is completely determined by the degrees of \(g\) on the connected components of \(\partial M\). Their sum is the total degree \(\deg(g)\). In particular, if \(\partial M\) is connected, the single integer \(\deg(g)\) determines \([g]\).
For \(d=1\), define \(\deg(g)\) as the sum of the
orientation signs of the boundary points that \(g\) maps to the positive
unit vector of \(E(V,r)\).
For an open or closed cover \(\mathcal{F}\) of \(\partial M\) with no convex \(r\)-balanced intersections, its degree is well defined by
\[
    \deg(\mathcal{F}_{(V,r)})=\deg(g),
    \qquad [g]=[\mathcal{F}_{(V,r)}].
\]

The following theorem follows from Theorem~4.5 of \cite{Blu}.

\begin{theorem}
\label{IndTh}
Let \((V,r)\) and \((V',r')\) be \(\conv\bs\)-equivalent pairs of common
rank \(d\geq1\). Let \(X\) be a nonempty subset of a Euclidean space, and
let \(\mathcal{F}=\{F_1,\dots,F_m\}\) be an open or closed cover of \(X\).
Then \([\mathcal{F}_{(V,r)}]=N/A\) if and only if
\([\mathcal{F}_{(V',r')}]=N/A\). When the classes are defined,
\([\mathcal{F}_{(V,r)}]\) is trivial if and only if
\([\mathcal{F}_{(V',r')}]\) is trivial.
If the classes are defined and \(X=\partial M\) for a compact
oriented \(d\)-manifold \(M\), then
\(|\deg(\mathcal{F}_{(V,r)})|
=|\deg(\mathcal{F}_{(V',r')})|\).
\end{theorem}

The next result is the ball case of Theorem~3.6 of
\cite{MusH}. It also follows from Theorem~2 of \cite{Blu}.

\begin{theorem}
\label{BallCoverCriterion}
Let \(\mathcal{F}=(F_i)_{i\in[m]}\) be an open or closed
cover of the closed ball \(B^{k+1}\), with boundary \(S^k\).
If the homotopy class \(\bigl[(\mathcal{F}|_{S^k})_{(V,r)}\bigr]\) is
defined and nontrivial, then there exists \(I\in\conv\bs(V,r)\) such that
\[
    \bigcap_{i\in I}F_i\neq\varnothing.
\]

\end{theorem}

Figure~\ref{fig:triangle-covers} illustrates two possible boundary
degrees for covers of a triangle.

\begin{figure}[htbp]
\centering
\begin{tikzpicture}[line join=round,every node/.style={font=\footnotesize}]
    \definecolor{ballcoverone}{HTML}{D2E5F5}
    \definecolor{ballcovertwo}{HTML}{FFE1B8}
    \definecolor{ballcoverthree}{HTML}{CBE7D7}
    \fill[ballcoverone]
        (-1.65,0)--(0,0)--(0,0.95)--(-0.825,1.425)--cycle;
    \fill[ballcovertwo]
        (1.65,0)--(0,0)--(0,0.95)--(0.825,1.425)--cycle;
    \fill[ballcoverthree]
        (0,2.85)--(-0.825,1.425)--(0,0.95)--(0.825,1.425)--cycle;
    \draw[black!80] (-1.65,0)--(1.65,0)--(0,2.85)--cycle;
    \draw[black!80] (0,0)--(0,0.95)--(-0.825,1.425)
        (0,0.95)--(0.825,1.425);
    \fill (0,0.95) circle[radius=1.3pt];
    \node at (-0.72,0.48) {\(F_1\)};
    \node at (0.72,0.48) {\(F_2\)};
    \node at (0,1.83) {\(F_3\)};
    \node[anchor=north,inner sep=0pt] at (0,-0.20)
        {\(\deg=1\) on the boundary};

    \begin{scope}[xshift=4.5cm]
        \fill[ballcoverone] (-1.65,0)--(-0.55,0)--(-0.55,1.9)--cycle;
        \fill[ballcovertwo]
            (-0.55,0)--(0.55,0)--(0.55,1.9)--(0,2.85)
            --(-0.55,1.9)--cycle;
        \fill[ballcoverthree] (1.65,0)--(0.55,0)--(0.55,1.9)--cycle;
        \draw[black!80] (-1.65,0)--(1.65,0)--(0,2.85)--cycle;
        \draw[black!80] (-0.55,0)--(-0.55,1.9) (0.55,0)--(0.55,1.9);
        \node at (-0.88,0.48) {\(F_1\)};
        \node at (0,0.48) {\(F_2\)};
        \node at (0.88,0.48) {\(F_3\)};
        \node[anchor=north,inner sep=0pt] at (0,-0.20)
            {\(\deg=0\) on the boundary};
        \node[anchor=north,inner sep=0pt] at (0,-0.58)
            {};
    \end{scope}
\end{tikzpicture}
\caption{Two covers of a triangle with boundary degrees \(1\) and \(0\),
respectively. ChatGPT (6.1) was used to generate this TikZ schematic.}
\label{fig:triangle-covers}
\end{figure}

The next lemma gives a way to construct covers with a prescribed homotopy class.
\begin{lemma}
\label{MaxCoordinateCover}
Let \(\mathcal{F}=(F_i)_{i\in[m]}\) be an open or closed
cover of a nonempty subset \(X\) of a Euclidean space, and let
\(h:X\to S(V,r)\) be continuous. Suppose that
\[
    x\in F_i\quad\Longrightarrow\quad
    \langle v_i,h(x)\rangle
    =\max_{j\in[m]}\langle v_j,h(x)\rangle.
\]
Then \([\mathcal{F}_{(V,r)}]=[h]\). 
\end{lemma}

\begin{proof}
After translating by \(-r\), we may assume \(r=0\).
Since \(0\in\operatorname{relint}\conv(V)\),
\(\max_j\langle v_j,h(x)\rangle>0\). Choose an open cover
\((G_i)\) with the same nerve, containing \((F_i)\), such that
\(G_i\subseteq\{x:\langle v_i,h(x)\rangle>0\}\), and let
\(\Phi\) be a subordinate partition of unity. These inequalities
exclude convex balanced intersections and give
\(\langle\rho_{\Phi,V}(x),h(x)\rangle>0\). Thus
\(g_{\Phi,V,0}(x)\) lies in the open hemisphere centered at \(h(x)\),
and the shortest geodesics give the required homotopy.
\end{proof}

\subsection{Homotopically nontrivial covers of Euclidean spaces}

Let \(E\) be a Euclidean space, for example an affine subspace of
\(\mathbb{R}^N\), and let \(\mathcal{F}=(F_i)_{i\in[m]}\) be an open or
closed cover of \(E\).
For a continuous map \(f:B^k\to E\) from the closed \(k\)-ball, the
pullback cover is \(f^*\mathcal{F}=(f^{-1}(F_i))_{i\in[m]}\).

\begin{definition}
The cover \(\mathcal{F}\) is \emph{convex balanced intersecting}
with respect to \((V,r)\) if \(\bigcap_{i\in I}F_i\neq\emptyset\) for some
\(I\in\conv\bs(V,r)\).
\end{definition}

\begin{definition}
\label{EucNomNontriv}
The cover \(\mathcal{F}\) is \emph{homotopically nontrivial with respect
to \((V,r)\)} if there exist \(k\geq1\) and a continuous map
\(f:B^k\to{E}\) such that the
boundary pullback cover \((f^*\mathcal{F})|_{\partial B^k}\)
has a defined, nontrivial homotopy class.
\end{definition}

The following result is the defined-class case of the forward implication
in \cite[Theorem~5.2]{Blu}.

\begin{theorem}
\label{NTC}
Let \(\mathcal{F}=\{F_1,\dots,F_m\}\) be an open or closed cover of a
Euclidean space \(E\). If \(\mathcal{F}\) is homotopically
nontrivial with respect to \((V,r)\), then it is convex balanced intersecting
with respect to \((V,r)\).
\end{theorem}

\subsection{Local index}

Let \(E\) be a \(d\)-dimensional Euclidean space, and let
\(\mathcal{F}=\{F_1,\dots,F_m\}\) be a closed cover of \(E\). In this
subsection we assume that the pair \((V,r)\) has rank \(d\geq1\).
Fix an orientation of \(E\), and use the chosen orientation
of \(E(V,r)\).
Every compact \(d\)-dimensional submanifold \(M\subset E\) inherits the orientation of
\(E\), and \(\partial M\) has the outward-normal-first orientation.

Define the convex balanced intersection set
\[
    Z_{(V,r)}(\mathcal{F})
    =
    \bigcup_{I\in\conv\bs(V,r)}
    \bigcap_{i\in I}F_i.
\]
The cover \(\mathcal{F}\) is called \emph{isolated} if every connected
component of \(Z_{(V,r)}(\mathcal{F})\) is compact and has a neighborhood
which does not meet any other connected component.

Assume that the cover is isolated. Let \(C\) be a connected component
of \(Z_{(V,r)}(\mathcal{F})\). Choose a
compact \(d\)-manifold \(M\subset{E}\) with smooth boundary such that
\(C\subset\operatorname{int}(M)\) and
\(M\cap Z_{(V,r)}(\mathcal{F})=C\). The restricted cover on
\(\partial M\) has no convex balanced intersections, so its degree is defined.

\begin{definition}
The \emph{index} of \(C\) with respect to \((V,r)\) is
\[
    \operatorname{ind}_{(V,r)}(C)
    =
    \deg\!\left((\mathcal{F}|_{\partial M})_{(V,r)}\right).
\]

\end{definition}

The index is independent of the choice of \(M\) by
Proposition~5.4 and Theorem~5.5 of \cite{Blu}.
Moreover, if \((V',r')\) is \(\conv\bs\)-equivalent to \((V,r)\), then
\(Z_{(V,r)}(\mathcal{F})=Z_{(V',r')}(\mathcal{F})\), and for every
connected component \(C\) of this set one has
\(|\operatorname{ind}_{(V,r)}(C)|
=|\operatorname{ind}_{(V',r')}(C)|\).

The next theorem is Theorem~5.5 of \cite{Blu}.

\begin{theorem}
\label{IndexSum}
Let \(\mathcal{F}\) be a closed isolated cover of \(E\),
and let \(M\subset{E}\) be a compact \(d\)-manifold with smooth boundary
such that \(\partial M\cap Z_{(V,r)}(\mathcal{F})=\emptyset\). If
\(\deg((\mathcal{F}|_{\partial M})_{(V,r)})=k\),
then
\[
    \sum_C \operatorname{ind}_{(V,r)}(C)=k,
\]
where the sum is taken over all connected components
\(C\subset\operatorname{int}(M)\) of
\(Z_{(V,r)}(\mathcal{F})\).
\end{theorem}

Figure~\ref{fig:index-sum} illustrates the index sum for a triangular region.

\begin{figure}[htbp]
\centering
\begin{tikzpicture}[x=0.8cm,y=0.8cm,
    every node/.style={font=\small,text=black},
    line cap=round,line join=round]
    \definecolor{indexcoverone}{HTML}{D2E5F5}
    \definecolor{indexcovertwo}{HTML}{FFE1B8}
    \definecolor{indexcoverthree}{HTML}{CBE7D7}
    \fill[indexcovertwo] (-3,-2)--(3,-2)--(0,3.196152)--cycle;
    \fill[indexcoverone]
        (-3,-2)--(0,-2)--(0,0)
        .. controls (-0.65,0.5) and (-1.12,0.72) .. (-1.5,0.598076)
        --cycle;
    \fill[indexcoverthree]
        (0,3.196152)--(-1.5,0.598076)
        .. controls (-1.12,0.72) and (-0.65,0.5) .. (0,0)
        .. controls (0.3,0.9) and (0.3,1.7) .. (0.65,2.070319)
        --cycle;
    \fill[indexcoverone]
        (0,0).. controls (0.5,0.5) and (1.85,0.5) .. (1.35,0)--cycle;
    \fill[indexcoverthree]
        (0,0).. controls (0.5,-0.5) and (1.85,-0.5) .. (1.35,0)--cycle;
    \draw[black!80,line width=0.6pt]
        (0,-2)--(0,0)--(1.35,0)
        (0,0).. controls (-0.65,0.5) and (-1.12,0.72) .. (-1.5,0.598076)
        (0,0).. controls (0.3,0.9) and (0.3,1.7) .. (0.65,2.070319)
        (0,0).. controls (0.5,0.5) and (1.85,0.5) .. (1.35,0)
        (0,0).. controls (0.5,-0.5) and (1.85,-0.5) .. (1.35,0);
    \draw[black!80,line width=0.85pt]
        (-3,-2)--(3,-2)--(0,3.196152)--cycle;
    \fill[black!85] (0,0) circle[radius=1.4pt];
    \fill[black!85] (1.35,0) circle[radius=1.4pt];
    \node[anchor=east,fill=white,inner sep=1pt] at (-0.15,-0.12)
        {\(+2\)};
    \node[fill=white,inner sep=1pt] at (1.35,-0.55) {\(-1\)};
    \node at (-3.12,-2.17) {\(1\)};
    \node at (3.12,-2.17) {\(2\)};
    \node at (0,3.396152) {\(3\)};
\end{tikzpicture}
\caption{A KKMS cover with local indices \(2\) and \(-1\)
and boundary degree \(1\). ChatGPT (6.1) was used to generate this TikZ schematic.}
\label{fig:index-sum}
\end{figure}

\subsection{Borsuk--Ulam type generalization of KKMS theorem}
\label{sec:kkms-borsuk-ulam}

We first explain the KKMS theorem within our framework.
Let \(n\geq2\) and \(\Delta=\Delta^{n-1}\). Put
\[
    v_S=\frac{\mathbf{1}_S}{|S|},\qquad
    V=(v_S)_{\varnothing\neq S\subseteq[n]},\qquad
    r=v_{[n]}=\frac{\mathbf{1}_{[n]}}{n}.
\]
Thus \(v_S\) is the barycenter of the face \(\Delta_S\).
Let \(\mathcal F\) be a closed KKMS cover of \(\Delta\). Take an open
enlargement of \(\mathcal F|_{\partial\Delta}\) with the same nerve.
For each label \(S\), remove from the enlarged set the faces \(\Delta_T\)
with \(S\nsubseteq T\). By~\eqref{KKMSFaceCondition}, the resulting
enlargement still contains each original member and has the same nerve. For a
subordinate partition of unity \((\varphi_S)\), the map
\(x\mapsto\sum_S\varphi_S(x)v_S\) sends each proper face into itself.
Its straight-line homotopy to the identity stays on the boundary.
Radial projection from \(r\) therefore gives boundary degree \(1\),
using the orientation of \(E(V,r)\) on \(\Delta\). The KKMS theorem
then follows from Theorem~\ref{BallCoverCriterion}.
\begin{definition}
\label{BoundaryNestedKKMS}
A cover
\(\mathcal{F}=(F_S)_{\varnothing\neq S\subseteq[n]}\) of \(\Delta\)
is called \emph{nested on
\(\partial\Delta\)} if, for every \(x\in\partial\Delta\), the family
\(\{S:x\in F_S\}\) is totally ordered by inclusion. Equivalently,
\[
    F_A\cap F_B\cap\partial\Delta\neq\varnothing
    \quad\Longrightarrow\quad
    A\subseteq B\ \text{or}\ B\subseteq A.
\]

\end{definition}

We now introduce a Borsuk--Ulam type generalization of the
KKMS theorem.

\begin{theorem}
\label{TBD}
Let
\(\mathcal{G}=(G_S)_{\varnothing\neq S\subseteq[n]}\) be a closed KKMS
cover of \(\Delta\) that is nested on \(\partial\Delta\).
Let \(T:2^{[n]}\to2^{[n]}\) satisfy \(T(\varnothing)=\varnothing\) and \(T([n])=[n]\), and
suppose that, for every nonempty proper \(S\subsetneq[n]\), either
\[
    T(S)=S,\qquad T(S^{\mathrm{op}})=S^{\mathrm{op}},
\]
or
\[
    T(S)=S^{\mathrm{op}},\qquad T(S^{\mathrm{op}})=S.
\]
Define the closed cover
\(\mathcal{F}=(F_S)_{\varnothing\neq S\subseteq[n]}\) by
\(F_S=G_{T(S)}\); in particular, \(F_{[n]}=G_{[n]}\).
Then there exists a balanced family \(\Phi\) such that
\[
    \bigcap_{S\in\Phi}F_S\neq\varnothing.
\]

\end{theorem}

\begin{proof}
If \(\mathcal{F}\) has a balanced intersection on
\(\partial\Delta\), there is nothing to prove. Otherwise,
\(F_{[n]}\cap\partial\Delta=G_{[n]}\cap\partial\Delta=\varnothing\).

The vertices of the barycentric subdivision
\(\operatorname{sd}(\partial\Delta)\) are the points \(v_S\) indexed
by nonempty proper subsets of \([n]\), and its simplices correspond to
chains of subsets. Write \(\ell(z)=z/\|z\|\) for radial projection in
\(E(V,r)\), and put \(u_S=\ell(v_S-r)\). Projection from \(r\)
gives a spherical triangulation of \(S(V,r)\simeq S^{n-2}\) with
vertices \(u_S\).

Extend the label map \(T\) to a spherical map \(\widehat T\) by
requiring, on each chain \(S_1\subsetneq\cdots\subsetneq S_k\),
\[
    \widehat T\!\left(\ell\!\left(\sum_{j=1}^k a_j u_{S_j}\right)\right)
    =\ell\!\left(\sum_{j=1}^k a_j u_{T(S_j)}\right),
    \qquad a_j\geq0,\quad\sum_j a_j>0.
\]
The vectors of a chain are linearly independent, so their convex hull
omits \(0\). The sign changes \(u_{T(S_j)}=\pm u_{S_j}\) preserve linear
independence, so the resulting convex hull also omits \(0\). The formulas are therefore defined on the spherical
simplices and agree on common faces, giving a continuous map. Since
\(u_{S^{\mathrm{op}}}=-u_S\) and
\(T(S^{\mathrm{op}})=T(S)^{\mathrm{op}}\), the map \(\widehat T\)
is odd and hence has odd degree.

Choose a partition of unity \((\varphi_S)\) subordinate to an open
enlargement of \(\mathcal{G}|_{\partial\Delta}\) with the same nerve.
The same functions, reindexed by \(T\), give a partition for the
corresponding enlargement of \(\mathcal{F}|_{\partial\Delta}\).
Define the normalized boundary maps
\[
\begin{aligned}
    \rho_{\mathcal{G}}(x)
        &=\ell\!\left(\sum_S\varphi_S(x)u_S\right),\\
    \rho_{\mathcal{F}}(x)
        &=\ell\!\left(\sum_S\varphi_S(x)u_{T(S)}\right),
\end{aligned}
\]
where both sums run over nonempty proper \(S\subsetneq[n]\).
Positive rescaling from \(v_S-r\) to \(u_S\) does not change the cover
classes, so \(\deg\rho_{\mathcal{G}}=1\).

By nestedness, the labels with \(\varphi_S(x)>0\) form a single chain
at every boundary point. The defining formula for \(\widehat T\)
therefore gives
\[
    \rho_{\mathcal{F}}=\widehat T\circ\rho_{\mathcal{G}}.
\]
Thus \(\rho_{\mathcal{F}}\) has odd degree. Theorem~\ref{BallCoverCriterion} then
gives the result.
\end{proof}

\section{Proof of the main result}

Throughout this section, \((U,V,r)\) is a generalized
game as in Definition~\ref{GNTP}.

\subsection{The cover associated with a game}
\label{sec:game-cover}

We first construct the cover of \(H_0\simeq\mathbb{R}^{n-1}\). For
\(x\in\mathbb{R}^n\), define
\[
    \tau(x)=\sup\left\{t\in\mathbb{R}\mid
    x+t\mathbf{1}_{[n]}\in\bigcup_{i=1}^m U_i\right\}
\]
and put \(p(x)=x+\tau(x)\mathbf{1}_{[n]}\).
Thus, \(p(x)\) is the maximal feasible point on the line through \(x\)
parallel to \(\mathbf{1}_{[n]}\).

\begin{lemma}
\label{CoverInducing}
The indexed family \(U=(U_i)_{i\in[m]}\) defines a closed cover
\(F(U)=(F_i)_{i\in[m]}\) of \(H_0\), where
\(F_i=F(U_i)=\{x\in H_0\mid p(x)\in U_i\}\).
\end{lemma}

\begin{proof}
The function \(\tau\) is well-defined. Indeed, nonemptiness and
comprehensiveness of the sets \(U_i\) imply that
\(x+t\mathbf{1}_{[n]}\in\bigcup_i U_i\) for sufficiently small \(t\).
Since every \(U_i\) is proper and comprehensive, this fails for sufficiently
large \(t\). Since the sets \(U_i\) are closed, the supremum is attained.

Moreover, comprehensiveness gives
\(|\tau(x)-\tau(y)|\leq \|x-y\|_\infty\).
Hence \(\tau\), and therefore \(p\), is continuous. It follows that each
\(F_i=p^{-1}(U_i)\cap H_0\) is closed. Finally, \(p(x)\) belongs to some
\(U_i\) for every \(x\), so the sets \(F_i\) cover \(H_0\).
\end{proof}

\subsection{Fractional cores and balanced intersections}

\begin{proof}[Proof of Theorem~\ref{GenTheorem}]
Let \(x\in Z(U,V,r)\). Then \(p(x)\) belongs to all \(U_i\) from some
\(r\)-balanced family, and hence is admissible. Moreover, \(p(x)\)
cannot lie in the interior of any \(U_i\), since otherwise it could be moved
slightly farther in the direction \(\mathbf{1}_{[n]}\), contradicting the
definition of \(\tau(x)\). Thus \(p(x)\) belongs to the fractional core.

Conversely, let \(y\) belong to the fractional core and project it onto
\(H_0\) by setting
\(x=y-\frac{1}{n}\langle y,\mathbf{1}_{[n]}\rangle\mathbf{1}_{[n]}\).
Since \(y\) is admissible, \(p(x)\geq y\). If \(p(x)\neq y\), then \(p(x)\)
is strictly larger than \(y\) in every coordinate and belongs to some
\(U_i\). Comprehensiveness would imply that \(y\) lies in the interior of
\(U_i\), a contradiction. Hence \(p(x)=y\). The balanced family witnessing
the admissibility of \(y\) now shows that \(x\in Z(U,V,r)\).
The inverse of \(p\) on the fractional core is the displayed linear
projection onto \(H_0\); hence the correspondence is a homeomorphism.
\end{proof}

\subsection{Normalization of the firms}

To apply the topological results from Section 2, we normalize the resource
vectors. Put
\(\bar v_i=v_i/\langle v_i,\mathbf{1}_{[d]}\rangle\),
\(\bar r=r/\langle r,\mathbf{1}_{[d]}\rangle\), and
\(\bar V=(\bar v_i)_{i\in[m]}\).
All points of \(\bar V\), as well as \(\bar r\), lie in the hyperplane
\(\langle x,\mathbf{1}_{[d]}\rangle=1\). Moreover,
\(\bar r\in\operatorname{relint}\conv(\bar V)\).
Since \(\cone(V)\) is full dimensional,
\(A(\bar V,\bar r)\) is this hyperplane and the normalized pair has rank \(d-1\).

\begin{lemma}
\label{NormalizationLemma}
For every nonempty \(I\subset[m]\),
\[
 I\in\bs(V,r)\quad\Longleftrightarrow\quad
 I\in\conv\bs(\bar V,\bar r).
\]
Consequently, the games \((U,V,r)\) and \((U,\bar V,\bar r)\) have the
same fractional core.
\end{lemma}

\begin{proof}
Write \(a_i=\langle v_i,\mathbf{1}_{[d]}\rangle\) and
\(a=\langle r,\mathbf{1}_{[d]}\rangle\). If
\(r=\sum_{i\in I}\lambda_i v_i\) with \(\lambda_i\geq0\), then
\(\bar r=\sum_{i\in I}(\lambda_i a_i/a)\bar v_i\) and
\(\sum_{i\in I}\lambda_i a_i/a=1\).
The converse follows by reversing this calculation. Thus the balanced
families are the same. Since the sets \(U_i\) are unchanged, the
fractional cores are also the same.
\end{proof}

\begin{remark}
Two indexed families \(U=\{U_1,\dots,U_m\}\) and
\(U'=\{U'_1,\dots,U'_m\}\) are called \emph{cover-equivalent} if they
induce the same cover, that is, \(F(U_i)=F(U'_i)\) for every \(i\).
Two pairs \((V,r)\) and \((V',r')\) are called
\emph{balanced-equivalent} if they have the same balanced
sets of indices. Two games are called \emph{equivalent} if their allocation
families are cover-equivalent and their resource pairs are
balanced-equivalent.

Thus, for the fractional-core existence questions considered here, a game
is determined by the induced cover of
\(H_0\simeq\mathbb{R}^{n-1}\) and the family of its balanced sets.
Equivalent games have the same balanced intersections, and hence their
fractional cores are simultaneously empty or nonempty.
\end{remark}

\subsection{Homotopical nontriviality and the main results}

\begin{definition}
\label{GameHomNontriv}

The game is \emph{homotopically nontrivial} if \(F(U)\) is homotopically
nontrivial with respect to \((\bar V,\bar r)\). Otherwise, the game is
\emph{homotopically trivial}.
\end{definition}

\begin{proof}[Proof of Theorem~\ref{HomNontrivGameThm}]
If the game is homotopically nontrivial, Theorem~\ref{NTC}
shows that \(F(U)\) is convex balanced intersecting with respect to
\((\bar V,\bar r)\). By Lemma~\ref{NormalizationLemma}, the game is
balanced intersecting, so Theorem~\ref{GenTheorem} gives a nonempty
fractional core.
\end{proof}

\begin{proof}[Proof of Theorem~\ref{GenScTh}]
Every fractional-core point belongs to \(\bigcap_{i\in I}U_i\)
for some \(I\in\bs(V,r)\). By balancedness, it belongs to \(U_{m+1}\) and
hence to the core. Conversely, the singleton index set \(\{m+1\}\) is \(r\)-balanced, so every core
point is admissible and belongs to the fractional core. The two cores
therefore coincide. If the game is homotopically nontrivial,
Theorem~\ref{HomNontrivGameThm} guarantees a nonempty fractional core.
\end{proof}

\begin{proposition}
\label{GameEquivalenceInvariant}
Let \((U,V,r)\) and \((U',V',r')\) be equivalent games. Then one of them is
homotopically nontrivial if and only if the other one is homotopically
nontrivial.
\end{proposition}

\begin{proof}
Since the games are equivalent, they have the same induced cover on \(H_0\)
and \((V,r)\) is \(\bs\)-equivalent to \((V',r')\).
By Lemma~\ref{NormalizationLemma}, the normalized pairs
are \(\conv\bs\)-equivalent. If their common rank is zero, every boundary
cover has a convex balanced singleton, so both games are homotopically
trivial. For positive rank, by Theorem~\ref{IndTh} applied to the
boundary pullback covers, the induced cover
\(F(U)\) is homotopically nontrivial with respect to \((\bar V,\bar r)\)
if and only if it is homotopically nontrivial with respect to
\((\bar V',\bar r')\).
\end{proof}

\begin{proposition}
\label{CoverRealization}
Let \(\mathcal{F}=\{F_1,\dots,F_m\}\) be a finite closed cover of \(H_0\)
by nonempty sets. Let \(V=(v_i)_{i\in[m]}\) and \(r\) satisfy the
resource assumptions in the definition of a game. Then there exists a game
\((U,V,r)\) such that \(F(U)=\mathcal{F}\).
\end{proposition}

\begin{proof}[Proof sketch]
Put \(U_i=F_i-\mathbb{R}^n_+\). It can be verified directly that \(U_i\) is closed, nonempty, proper and comprehensive. It also follows directly that  \(F(U_i)=F_i\).
\end{proof}
 
\section{Examples and Discussion}

In this section, we consider examples and applications of Theorems
\ref{GenTheorem} and \ref{GenScTh}. We also describe
political games and cooperative games with threats in this framework.

\subsection{Scarf's theorem}

We begin by explaining how Theorem \ref{GenScTh} recovers Scarf's theorem.
Let
\(([n],V)\) be a standard cooperative NTU game, and let \(U(S)\) be the
cylindrical feasible set associated with a coalition \(S\). We regard each
nonempty coalition \(S\subset[n]\) as a firm, assign to it the resource
vector \(v_S=\mathbf{1}_S/|S|\), and put \(r=\mathbf{1}_{[n]}/n\).

Order the firms so that the last one is the grand coalition \(S=[n]\).
Then \(v_{m+1}=v_{[n]}=r\), and the distinguished set in Theorem
\ref{GenScTh} is
\[
    U_{m+1}=U([n])=V([n]).
\]
Thus the generalized core
\(U_{m+1}\setminus\bigcup_i\operatorname{int}(U_i)\) is exactly the usual
core of the NTU game.

Balanced families of coalitions are the same as \(r\)-balanced families of
the vectors \(v_S\), after a rescaling of the weights. Hence the usual
balancedness condition for an NTU game is the balancedness condition from
Theorem \ref{GenScTh}.

\begin{proof}[Proof sketch]
For \(n\geq2\) and large \(M\), take
\(Q_M=\conv\{-n^2Me_i:i\in[n]\}\) and \(F_S=\{x\in Q_M:p(x)\in U(S)\}\).
On the face indexed by \(T\), \(p(x)_j=\tau(x)\) for \(j\notin T\).
Since \(\sum_i x_i=-n^2M\) and \(p(x)_i\geq a_i\), we have
\(\tau(x)\geq nM+\frac{1}{n}\sum_i a_i\). This exceeds the bound on
individually rational feasible payoffs, so \(F_S\) misses the face if
\(S\nsubseteq T\). KKMS gives a balanced intersection whose frontier
payoff belongs to the fractional core and, if the game is balanced, to
the core. Translation along \(\mathbf{1}_{[n]}\) carries the cover into
\(H_0\), where its boundary degree has absolute value one. This verifies
the homotopy hypothesis of Theorem~\ref{GenScTh}. For \(n=1\), both
cores are \(\{a_1\}\).
\end{proof}

\subsection{Games of degree \texorpdfstring{\(k\)}{k}}

We now explain how the local index from Section 2 applies to generalized
games. In this subsection, the utility sets are subsets of
\(\mathbb{R}^{n+1}\) with \(n\geq1\). Hence the induced cover \(F(U)\) is a cover of
\(H_0\simeq\mathbb{R}^n\). We assume that the normalized
resource pair \((\bar V,\bar r)\) has rank
\(n\), and that \(F(U)\) is isolated with respect to this pair. Such a game
will be called \emph{isolated}.

By Theorem~\ref{GenTheorem}, every connected component \(C\) of the
fractional core corresponds to a connected component \(\widehat C\) of the
convex balanced intersection set of \(F(U)\). We define
\(\operatorname{ind}(C)=\operatorname{ind}_{(\bar V,\bar r)}(\widehat C)\).
The component \(C\) is called \emph{essential} if this index is nonzero,
and \emph{inessential} otherwise.

An isolated game is called \emph{essential} if at least one connected
component of its fractional core is essential. It is called
\emph{regular} if every connected component \(C\) of its fractional core
satisfies \(|\operatorname{ind}(C)|=1\).

In this terminology, every isolated classical NTU game with at least
two players is essential because the KKMS boundary cover
has absolute degree \(1\).

\begin{definition}[Degree of a game]
Let \((U,V,r)\) be an isolated game and let \(B_M\subset H_0\)
be the closed ball of radius \(M\) centered at the origin. If, for all
sufficiently large \(M\), the degree
\(\deg((F(U)|_{\partial B_M})_{(\bar V,\bar r)})\) is defined and equal to
\(k\neq0\), then the game has \emph{degree \(k\)}.
\end{definition}

For a game of degree \(k\), the balanced intersection set
is contained in a sufficiently large ball. By Theorem~\ref{IndexSum},
\(k\) is the sum of the indices of all connected components of the
fractional core.

\begin{theorem}
The fractional core of a regular game of degree \(k\)
has at least \(|k|\) connected components.
\end{theorem}

\begin{proof}
For sufficiently large \(M\), Theorem~\ref{IndexSum}
gives \(k\) as the sum of the indices of the balanced intersection
components inside \(B_M\). Each summand is \(\pm1\) by regularity, so
there are at least \(|k|\) such components. By Theorem~\ref{GenTheorem},
they correspond to distinct connected components of the fractional core.
\end{proof}

\subsection{The Hopf Fibration}

We now give an example of a game whose associated cover gives rise to a
homotopy invariant different from the mapping degree. Let
\(V=\{v_1,v_2,v_3,v_4\}\) be the vertices of a tetrahedron, which we take to
be the unit basis vectors in \(\mathbb{R}^4\), and let
\(r_0=\frac{1}{4}(v_1+v_2+v_3+v_4)\). The pair \((V,r_0)\) is already
normalized and has rank \(3\). The boundary
\(\partial\conv(V)\) is a two-dimensional sphere, and radial projection
from \(r_0\) identifies it with \(S(V,r_0)\).

Choose a Hopf map \(h:S^3\to S(V,r_0)\simeq S^2\)
with Hopf invariant \(1\). Define
\[
 F_i=\{x\in S^3:h_i(x)=\max_j h_j(x)\},\qquad i=1,2,3,4.
\]
By Lemma~\ref{MaxCoordinateCover}, this cover has the
nontrivial class \([h]\). Its total intersection is empty, and each
\(F_i\) is nonempty by surjectivity of \(h\). One can also start with a
simplicial Hopf map to the boundary of a tetrahedron \cite{Madahar,MadSar}.
Label each domain vertex by its image and take the union of open stars
for each label to obtain a cover with the same class
\cite[Example~1.3]{MusH}.

Embed \(S^3\) as the unit sphere in
\(H_0\subset\mathbb{R}^5\), with \(H_0\simeq\mathbb{R}^4\).
Let \(\mathcal{G}=(G_i)_{i=1}^4\) be any closed cover
of \(H_0\) with \(G_i\cap S^3=F_i\). Such a cover can be obtained by
choosing any closed extension to the unit ball \(B^4\) and extending
it conically outside the ball.

By Theorem~2 of \cite{Blu}, the restriction of
\(\mathcal{G}\) to \(B^4\) has a convex \(r_0\)-balanced intersection.
The only convex \(r_0\)-balanced index set is \(\{1,2,3,4\}\), so
all four sets \(G_i\) meet at a point in the interior of \(B^4\).

By Proposition~\ref{CoverRealization}, the feasible sets
\(U_i=G_i-\mathbb{R}^5_+\) define a game with this induced cover.
This common point gives a fractional-core point
by Theorem~\ref{GenTheorem}.

This example shows that fractional-core existence can be forced by a higher
homotopy class; here the class lies in \(\pi_3(S^2)\) and has
Hopf invariant \(1\).

\subsection{Homotopically trivial games}

Let \(\mathcal{F}\) be a finite open or closed cover of
\(H_0\simeq\mathbb{R}^{n-1}\), and let \((V,r)\) have rank
\(q\geq n\geq2\). For a continuous map \(f:B^k\to H_0\) with \(k\geq1\),
the homotopy class of the boundary pullback cover
\((f^*\mathcal{F})|_{\partial B^k}\) with respect to \((V,r)\),
when defined, is represented by
\(g\circ f|_{\partial B^k}\), where \(g\) represents the restricted
cover on the compact image \(X=f(\partial B^k)\). Since \(n-1\leq q-1\),
the map \(g:X\to S(V,r)\simeq S^{q-1}\) extends over the contractible
space \(H_0\) \cite[Theorems~1.8.2 and~1.9.3]{EngelkingDim}, so the
boundary class is trivial. This applies to games with \(d>n\geq2\),
whose normalized resource pair has rank \(d-1\geq n\).

The following theorem shows that small perturbations can make the
fractional core empty when \(d>n\).

\begin{theorem}
\label{HighResourceDimension}
Let \((U,V,r)\) be a generalized game with payoff space
\(\mathbb{R}^n\) and resource space \(\mathbb{R}^d\), where \(d>n\).
For every \(\varepsilon>0\), there exist a resource target \(r'\) with \(\|r'-r\|<\varepsilon\) and
shifts \(s_i\) with \(|s_i|<\varepsilon\) such that the game
\((U',V,r')\), with \(U_i'=U_i+s_i\mathbf{1}_{[n]}\), has an empty
fractional core.
\end{theorem}

\noindent\emph{Proof sketch.}
First choose \(r'\) arbitrarily close to \(r\), outside
the spans of all subfamilies of fewer than \(d\) resource vectors.
Then every \(r'\)-balanced family contains at least \(d\) firms.

We can choose arbitrarily small diagonal shifts
\(U_i'=U_i+s_i\mathbf{1}_{[n]}\) so that at most \(n\) members \(F(U_i')\)
of the shifted cover meet at any point of \(H_0\).
Indeed, for each of the finitely many choices of \(n+1\) firms \(i_0,\ldots,i_n\),
a tie requires \((s_{i_j}-s_{i_0})_{j=1}^n\) to lie in the image of the
Lipschitz map \(x\mapsto(\tau_{i_0}(x)-\tau_{i_j}(x))_{j=1}^n\) on \(H_0\),
where \(\tau_i(x)=\max\{t:x+t\mathbf{1}_{[n]}\in U_i\}\), so the exceptional
shift vectors form a set of Lebesgue measure zero because \(\dim H_0=n-1\)
and the linear map from shifts to shift differences is surjective.
Since \(d>n\), the shifted cover has no \(r'\)-balanced intersections, so the
fractional core of \((U',V,r')\) is empty.

\subsection{Political games}
Consider a parliament with parties indexed by \([m]\) and
seats indexed by \([d]\). Party \(i\) can occupy a specified nonempty set
\(S_i\subseteq[d]\) of seats. A \emph{working coalition} is a set of
parties whose seat sets partition \([d]\).

More generally, let \(r\in\mathbb{R}^d\) represent the
total power in parliament and \(v_i\in\mathbb{R}^d\) the power of party
\(i\). A set \(I\subseteq[m]\) is a working coalition if
\(r=\sum_{i\in I}v_i\). As before, we can define fractional working coalitions. Assume also that the power
vectors and \(r\) satisfy the resource conditions in
Definition~\ref{GNTP}.

We assume that the space of all possible policies is \(\mathbb{R}^n\). Each party \(i\) has a set
\(U_i\subseteq\mathbb{R}^n\) of feasible policies that we assume to be closed, proper, nonempty and comprehensive. Party \(i\) can block a
policy \(x\) if \(x\in\operatorname{int}(U_i)\).

When a working coalition exists, we take the grand
coalition's feasible policy set to be the supremum under inclusion of
the policy sets attainable by integer allocations of power,
\[
    U_{m+1}=\bigcup_{\substack{I\subseteq[m]\\ \sum_{i\in I}v_i=r}}
    \bigcap_{i\in I}U_i.
\]

We set \(v_{m+1}=r\) for the grand firm and define the core
of the political game as the core of the resulting generalized game.

\subsection{Cooperative games with threats}
\label{sec:threats}

Cooperative games with threats provide another natural
application of our framework. In a standard cooperative game, a stable
allocation is feasible for the grand coalition and cannot be blocked by
any coalition. The core consists precisely of these stable allocations.
The grand coalition is usually assumed to perform at least as well as
any partition of the players into coalitions.

Now suppose that coalitions can wage war, issue threats,
or impose vetoes on one another. For example, a player \(i\) abandoned
by the others may threaten them with a total loss of one unit, so their
coalition must generate a payoff of at least one unit to operate.
Now an allocation \(x\) is stable if it is feasible for the grand coalition
and cannot be blocked or threatened by any coalition.

Cooperative threats have been studied in various papers and in various models. For example, in models of bargaining and
settlement \cite{MyersonThreats,HartMasColell,VohraRay} and in the theory of cooperative values \cite{KohlbergNeymanThreats,KohlbergNeymanStrategic}.
Related approaches address coalition power through effectivity functions
\cite{MoulinPeleg,MizutaniLeeNishino}, interference by outsiders with
blocking coalitions \cite{PlattSpoilers}, and coercive power
\cite{JordanPillage}. Coalitional externalities are studied in partition
function form games and recursive core models \cite{ThrallLucas,KoczyRecursive}. Power vectors appear in the theory of socially structured games
\cite{HeringsLaanTalman}.

In our model, each nonempty coalition has a feasible payoff set.
Each nonempty proper coalition also has a separately specified set of
threats it can impose on its complement's payoffs. The examples below
show how this additional blocking power affects the fractional core and
the degree of the associated cover.

\begin{definition}
A \emph{cooperative game with threats} is a triple \(([n],V,T)\), where
\([n]\) is the player set,
\(V=(V(S))_{\varnothing\neq S\subseteq[n]}\), and
\(T=(T(S))_{\varnothing\neq S\subsetneq[n]}\).
The payoff and threat sets satisfy the following conditions.
\begin{enumerate}
    \item For each nonempty coalition \(S\), the feasible payoff set
    \(V(S)\subseteq\mathbb{R}^S\) is nonempty, proper, closed, and
    comprehensive in \(\mathbb{R}^S\).

    \item For each \(i\in[n]\), write \(V(\{i\})=(-\infty,a_i]\), where
    \(a_i=\max V(\{i\})\). For each nonempty coalition \(S\), the
    individually rational part
    \[
        V(S)\cap\{x\in\mathbb{R}^S:x_i\geq a_i\text{ for all }i\in S\}
    \]
    is bounded.

    \item For each nonempty proper coalition \(S\), the set
    \(T(S)\subseteq\mathbb{R}^{S^{\mathrm{op}}}\) describes the threats
    or vetoes that \(S\) can impose on its complement's payoffs.
    It is proper, closed, and comprehensive in
    \(\mathbb{R}^{S^{\mathrm{op}}}\) and may be empty. The part of \(T(S)\)
    at or above the singleton payoff levels,
    \[
        T(S)\cap\{x\in\mathbb{R}^{S^{\mathrm{op}}}:
        x_i\geq a_i\text{ for all }i\in S^{\mathrm{op}}\},
    \]
    is bounded.
\end{enumerate}
For each nonempty proper coalition \(S\), define \(U(S)\subseteq\mathbb{R}^n\)
as the union of the payoff and threat cylinders
\[
    U(S)=
    \bigl(V(S)+\mathbb{R}^{S^{\mathrm{op}}}\bigr)
    \cup\bigl(\mathbb{R}^S+T(S)\bigr).
\]
Set \(U([n])=V([n])\) and
\(T(\varnothing)=V(\varnothing)=U(\varnothing)=\varnothing\).
\end{definition}

An allocation \(x\in\mathbb{R}^{[n]}\) is blocked or threatened by coalition \(S\) if \(x\in\operatorname{int}U(S)\). As before, the core consists of stable allocations feasible for the grand coalition:
\[
    C:=V([n])\setminus
    \bigcup_{\varnothing\neq S\subseteq[n]}\operatorname{int}U(S).
\]
Admissibility and the fractional core are defined using
the sets \(U(S)\). The game \(([n],V,T)\) is balanced if, for every
balanced family \(\mathcal B\subseteq2^{[n]}\setminus\{\varnothing\}\),
\[
    \bigcap_{S\in\mathcal B}U(S)\subseteq V([n]).
\]

The covers induced by cooperative games with threats can have different boundary degrees.

\begin{exmp}
Every classical NTU game \(([n],V)\) is recovered by taking
\(T(S)=\varnothing\) for every nonempty proper coalition \(S\).
If \(V(S)\subseteq\operatorname{int}T(S^{\mathrm{op}})\) for every
nonempty proper coalition \(S\), only threats can be active at the
frontier for proper coalitions.
\end{exmp}

\begin{exmp}
Consider the zero TU game on \([3]\). For each nonempty proper coalition
\(S\), set \(T(S)=\{x\in\mathbb{R}^{S^{\mathrm{op}}}:
\sum_{i\notin S}x_i\leq t_S\}\), where \(t_{\{3\}}=1\) and
\(t_S=-1\) otherwise. An unblocked payoff \(x\) satisfies \(x\geq0\) and
\(x_1+x_2\geq1\). The only threat cylinder that can contain such an \(x\)
is the one labeled \(\{3\}\). A balanced family supporting \(x\) must
therefore use production sets to cover players \(1\) and \(2\), forcing
\(x_1=x_2=0\). Thus the fractional core is empty, and the induced cover
has boundary degree zero on every triangle in \(H_0\) by
Theorems~\ref{GenTheorem} and~\ref{BallCoverCriterion}.
\end{exmp}

\begin{definition}
\label{StrictThreatOrdinalConvexity}
For a game with threats \(([n],V,T)\), define the effective payoff sets by
\[
    W(N)=V(N),\qquad
    W(S)=V(S)\cup T(S^{\mathrm{op}})
    \quad(\varnothing\neq S\subsetneq N).
\]
 We call \(([n], W)\) the \emph{associated effective game}. Write
\(\bar U(S)=W(S)+\mathbb{R}^{S^{\mathrm{op}}}\) for its cylinders when
\(S\neq\varnothing\), and set \(\bar U(\varnothing)=\varnothing\).
\end{definition}

The associated effective game \(([n],W)\) satisfies the NTU
axioms in Definition~\ref{NTU}. Each \(W(S)\) is proper, closed, nonempty, and
comprehensive. The singleton payoff
levels of \(W\) are at least \(a_i\), so the boundedness assumptions on
\(V\) and \(T\) also bound the individually rational part of each \(W(S)\).

\begin{definition}
The game \(([n], W)\) is called \emph{strictly ordinally convex on proper unions} if, for
every pair of incomparable nonempty coalitions \(A,B\) with
\(A\cup B\subsetneq N\),
\[
    \bar U(A)\cap \bar U(B) \subseteq
    \operatorname{int}\bar U(A\cup B)\cup
    \operatorname{int}\bar U(A\cap B).
\]
Here incomparable means that neither coalition contains the other. A game
with threats is \emph{strictly ordinally convex on proper unions} if its
associated effective game has this property.
\end{definition}

The usual ordinal-convexity condition uses
\(\bar U(A\cup B)\cup\bar U(A\cap B)\) on the right and applies to all
pairs of coalitions \cite[assumption (A.7)]{GreenbergConvex}
\cite[Definition~2.7]{CsokaHeringsKoczyPinter}. Here we use interiors
and restrict the condition to incomparable pairs with proper union.

\begin{exmp}
Again, consider three players \([3]=\{1,2,3\}\).
Let \(([3],V,T)\) be a game with threats that is strictly ordinally convex
on proper unions, and suppose that its
production family \(V\) is a standard NTU game.
For every \(i\in [3]\), assume that
\[
    T([3]\setminus\{i\})\subseteq\operatorname{int}V(\{i\}),
    \qquad
    V([3]\setminus\{i\})\subseteq\operatorname{int}T(\{i\}).
\]
Thus each player is protected from the opposing pair's threats and can
strictly threaten every allocation feasible for that pair.

On a sufficiently large simplex, the boundary cover induced by the game
\(([3],V,T)\) has degree of absolute value \(2\), and the
fractional core is nonempty. Figure~\ref{fig:threat-boundary} illustrates
the boundary relabeling.

\begin{figure}[H]
\centering
\begin{tikzpicture}[every node/.style={font=\small,text=black},
    line cap=round,line join=round]
    \begin{scope}
        \node at (1.8,3.8) {Associated game};
        \draw[gray!65,line width=3pt]
            (0,0)--(3.6,0)--(1.8,3.12)--cycle;
        \draw[blue!70!black,line width=3pt]
            (0.45,0.78)--(0,0)--(0.9,0);
        \draw[orange!85!black,line width=3pt]
            (2.7,0)--(3.6,0)--(3.15,0.78);
        \draw[teal!75!black,line width=3pt]
            (2.25,2.34)--(1.8,3.12)--(1.35,2.34);
        \fill[black] (0,0) circle (0.7pt);
        \fill[black] (3.6,0) circle (0.7pt);
        \fill[black] (1.8,3.12) circle (0.7pt);
        \fill[black] (1.8,0) circle (0.7pt);
        \fill[black] (2.7,1.56) circle (0.7pt);
        \fill[black] (0.9,1.56) circle (0.7pt);
        \node[below left] at (0,0) {\(1\)};
        \node[below right] at (3.6,0) {\(2\)};
        \node[above] at (1.8,3.12) {\(3\)};
        \node[below] at (1.8,0) {\(12\)};
        \node[right] at (2.7,1.56) {\(23\)};
        \node[left] at (0.9,1.56) {\(13\)};
    \end{scope}
    \draw[->,black,thick] (4.5,1.4)--(5.9,1.4);
    \node[align=center] at (5.2,2.2)
        {\(12\mapsto3\)\\\(23\mapsto1\)\\\(13\mapsto2\)};
    \begin{scope}[xshift=6.8cm]
        \node at (1.8,3.8) {Game with threats};
        \draw[gray!65,line width=3pt]
            (0,0)--(3.6,0)--(1.8,3.12)--cycle;
        \draw[blue!70!black,line width=3pt]
            (0.45,0.78)--(0,0)--(0.9,0);
        \draw[blue!70!black,line width=3pt]
            (3.15,0.78)--(2.25,2.34);
        \draw[orange!85!black,line width=3pt]
            (2.7,0)--(3.6,0)--(3.15,0.78);
        \draw[orange!85!black,line width=3pt]
            (1.35,2.34)--(0.45,0.78);
        \draw[teal!75!black,line width=3pt]
            (2.25,2.34)--(1.8,3.12)--(1.35,2.34);
        \draw[teal!75!black,line width=3pt]
            (0.9,0)--(2.7,0);
        \fill[black] (0,0) circle (0.7pt);
        \fill[black] (3.6,0) circle (0.7pt);
        \fill[black] (1.8,3.12) circle (0.7pt);
        \fill[black] (1.8,0) circle (0.7pt);
        \fill[black] (2.7,1.56) circle (0.7pt);
        \fill[black] (0.9,1.56) circle (0.7pt);
        \node[below left] at (0,0) {\(1\)};
        \node[below right] at (3.6,0) {\(2\)};
        \node[above] at (1.8,3.12) {\(3\)};
        \node[below] at (1.8,0) {\(3\)};
        \node[right] at (2.7,1.56) {\(1\)};
        \node[left] at (0.9,1.56) {\(2\)};
    \end{scope}
\end{tikzpicture}
\caption{Boundary labels for the associated effective game and the game
with threats. ChatGPT (6.1) was used to generate this TikZ schematic.}
\label{fig:threat-boundary}
\end{figure}

The relabeled boundary map follows the shorter arcs between consecutive
labels in the sequence \(1,3,2,1,3,2,1\). Taking \(1,2,3\) as the positive
cyclic order on the target circle, the map makes two turns in the reverse
direction and has winding number \(-2\).

\end{exmp}

After considering these examples, we are ready to formulate a structural Scarf-type result for cooperative games with threats.

\begin{definition}
A game with threats \(([n],V,T)\) has
\emph{strong dominance symmetry} if, for every nonempty proper coalition
\(S\subsetneq[n]\), either
\[
    V(S)\subseteq\operatorname{int}T(S^{\mathrm{op}}),\qquad
    V(S^{\mathrm{op}})\subseteq\operatorname{int}T(S),
\]
or
\[
    T(S^{\mathrm{op}})\subseteq\operatorname{int}V(S),\qquad
    T(S)\subseteq\operatorname{int}V(S^{\mathrm{op}}).
\]
Interiors are taken in the corresponding coordinate spaces.
\end{definition}

We now obtain a Scarf-type theorem for games with threats satisfying
strict ordinal convexity on proper unions and strong dominance symmetry.

\begin{theorem}
Every game with threats that is strictly ordinally convex on proper
unions and has strong dominance symmetry has a nonempty fractional core. In
particular, if such a game is balanced,
then it has a nonempty core.
\end{theorem}

\begin{proof}[Proof sketch]
For \(n=1\), the maximum of \(V([1])\) belongs to the core.
For \(n\geq2\), let \(\mathcal{D}=(D(S))\) be the cover of \(H_0\)
induced by the associated game \(([n],W)\). On the boundary of a
sufficiently large simplex \(\Delta\subset H_0\) constructed as in
Section~4.1, this cover satisfies the KKMS face condition. If \(x\) lies on
a proper face \(\Delta_T\) and \(x\in D(A)\cap D(B)\), then \(A,B\subseteq T\).
Thus \(A\cup B\subsetneq[n]\). If \(A,B\) were
incomparable, strict ordinal convexity on proper unions would put
\(p(x)\) in the interior
of an effective cylinder, contradicting the definition of \(p(x)\).
Thus \(\mathcal{D}|_\Delta\) is a KKMS cover nested on \(\partial\Delta\). Let
\(\mathcal{F}=(F(S))\) be the cover induced by \(([n],V,T)\).
The two games have the same union of cylinders, so by strong dominance  either
\[
    \begin{aligned}
    &D(S)=F(S),\qquad D(S^{\mathrm{op}})=F(S^{\mathrm{op}}),\\
    &\text{or}\quad
    D(S)=F(S^{\mathrm{op}}),\qquad D(S^{\mathrm{op}})=F(S).
    \end{aligned}
\]
Also, \(D([n])=F([n])\). Theorem~\ref{TBD} now gives a balanced
intersection of \(\mathcal{F}|_\Delta\).
Theorem~\ref{GenTheorem} then gives a nonempty fractional core.
If the game is balanced, its core and fractional core coincide, as in
the proof of Theorem~\ref{GenScTh}.
\end{proof}

\section{Further Research}
\label{sec:further}

The results of this paper suggest several directions for further research.

\subsection{Combinatorial applications}

Scarf's lemma has important applications in combinatorics and matching
theory; see, for example, \cite{AharoniFleiner}. It would be interesting to
find analogous applications of Theorem \ref{GenTheorem}, in particular to
generalized matching problems and other combinatorial models of stability.

\subsection{Applications in game theory and economics}

The examples above show that generalized games may give rise to covers of
arbitrary degree and even to higher homotopy classes. It would be interesting
to find concrete games or economic situations that can be naturally and
precisely described by generalized cooperative games and whose associated
covers represent nontrivial higher homotopy classes.

\section*{Acknowledgments}

We are grateful to Hervé Moulin, Fedor Sandomirskii, Nikita Kalinin and Alexander Levin for reading the initial draft of the paper and for helpful comments. We thank the International Laboratory of Game Theory and Decision Making for helpful discussions.

\section*{Funding}

The first author was supported by the HSE University Basic Research Program.
The funder had no role in the research, preparation of the manuscript,
or the decision to submit it for publication.

\section*{Declaration of competing interests}

The authors declare that they have no competing interests.

\section*{Data availability}

No datasets were generated or analyzed for this theoretical study.

\section*{Declaration of generative AI and AI-assisted technologies in the manuscript preparation process}

The authors used ChatGPT (6.1) to assist with checking proofs, searching
the literature, improving the exposition, and generating the TikZ
schematics. The authors reviewed and edited the content as needed and
take full responsibility for the published article.

\end{document}